\documentclass[12pt,a4paper]{amsart}
\usepackage{amsmath, amssymb, amsfonts, graphicx}
\newtheorem{Lma}{Lemma}
\newtheorem{Thm}{Theorem}
\newtheorem{Prop}{Proposition}
\newtheorem*{MThm}{Main Theorem}
\title[Level set flow in 3D steady gradient Ricci solitons]{Level set flow in 3D steady gradient Ricci solitons}
\author{Chih-Wei Chen*}
\address{* National Center for Theoretical Sciences, Taiwan}
\email{BabbageTW@gmail.com}

\author{Kuo-Wei Lee**}
\address{** Department of Mathematics National Changhua University of Education, Taiwan}
\email{kwlee@cc.ncue.edu.tw; d93221007@gmail.com}


\begin{document}
\maketitle
\begin{abstract}
Let $(M^3, g, f)$ be a nontrivial 3-dimensional steady gradient Ricci soliton.
If the scalar curvature $R$ satisfies
$c_1r^{-b}\leq R\leq c_2r^{-a}$ for some $a\in(0,1], b\geq a$,
and $c_1,c_2>0$, then the umbilical ratio of the level sets of $f$ satisfies $\frac{2|A|^2-H^2}{H^2}\in O(r^{6a-\frac{8a^2}{b}})\cap O(r^{2b-4a})$.
\end{abstract}
\section{Introduction}
The characterization of gradient Ricci soliton is one of the central themes in the study of the Ricci flow with surgery.
For instance, the proof of Poincar\'e conjecture relies on the classification of shrinking solitons
and the asymptotic analysis of ancient solutions.
Indeed, along this line, G. Perelman \cite{Perelman02} said

\begin{quote}
``... {\it I believe that there is only one (up to scaling) noncompact three-dimensional
$\kappa$-noncollapsed ancient solution with bounded positive curvature --- the rotationally symmetric gradient steady soliton,
studied by R.Bryant. In this direction, I have a plausible, but not quite rigorous argument, showing that any such ancient solution can be made eternal,
that is, can be extended for $t\in(-\infty,\infty)$; also I can prove uniqueness in the class of gradient steady solitons.}"
\end{quote}

It is still unknown whether a $3$-dimensional $\kappa$-noncollapsed ancient solution with bounded positive curvature is the Bryant soliton.
However, when restricting to the class of gradient steady solitons,
S. Brendle gave a positive answer to the uniqueness problem \cite{Brendle13}, which Perelman claimed that he had had a proof in mind.
The noncollapsing assumption arises naturally because Perelman showed that the Ricci flow is
$\kappa$-noncollapsed when encountering a finite time singularity {\it somewhere on the manifold}.
An example of which the singularity arises at spatial infinity was constructed by T. Richard and the first author (cf. \cite{ChenThesis}).
Namely, the flow remains regular at the maximal time $T$ and the curvature $|Rm|(x,T)\to\infty$ as $x\to\infty$.
Although the example is $\kappa$-noncollapsed and of type-I, it seems possible to have collapsed singularities in general.
Such singularities might infinitesimally look like any bundle with Hamilton's cigar soliton as its fibers.

The collapsing phenomenon is much less understood not due to the lack of importance, but due to the lack of mechanism.
Many of Perelman's arguments are based on the blow-up method and thus fail on the collapsing case.
Not to mention collapsed singularities, we know very few even for collapsed steady gradient Ricci solitons,
which serve as models of singularities. For instance, the non-existence of a positively curved 3-dimensional
$\kappa$-collapsed steady gradient Ricci soliton is still unsettled.
In this article, we propose the level set flow to study $\kappa$-collapsed steady gradient Ricci solitons.
Namely, we look at the level sets of the potential function $f$ and treat them as a family of evolving hypersurfaces parametrized by the value of $f$.
The evolution of these hypersurfaces, which is confined by a weighted mean curvature flow,
may determine the shape of the ambient manifold --- the steady soliton we would like to investigate.
Our main result is
\begin{MThm}
Let $(M^3, g, f)$ be a nontrivial 3-dimensional steady gradient Ricci soliton.
If the scalar curvature $R$ satisfies
\begin{align}\label{ab}
c_1 r^{-b}(x)\leq R(x)\leq c_2 r^{-a}(x)
\end{align}
for some $a\in(0,1], b\geq a$, $c_1,c_2>0$, and $r(x):=dist(O,x)$ for some fixed $O\in M$,
then the soliton satisfies
\begin{align*}
\frac{2|A|^2-H^2}{H^2}\in O\left(r^{6a-\frac{8a^2}{b}}\right)\cap O\left(r^{2b-4a}\right),
\end{align*}
where $|A|^2$ and $H$ are the norm square of the second fundamental form and the mean curvature of the level sets of $f$, respectively. In particular, when $b\in[a,2a)$, the soliton is asymptotically round in the sense that the umbilical ratio $\frac{2|A|^2-H^2}{H^2}$ vanishes at infinity. 
\end{MThm}

There have been two ways to obtain the asymptotic roundness of three-dimensional steady gradient Ricci solitons.
However, both of them use convergence argument, which needs the $\kappa$-noncollapsity.
One is Perelman's theorem which says that the blow-down limit must be a shrinking soliton,
and thus a cylinder. The other way given by Y. Deng and X. Zhu is based on dimension reduction and the classification of
closed two-dimensional ancient solutions \cite{DengZhu16}.
Precisely, they proved that, when $a=b=1$,
such a steady 3-dimensional gradient Ricci soliton has a uniform lower bound for the volume ratio of remote geodesic balls with linear growth radius.
Thus they are able to show the soliton is (subsequentially) asymptotically
cylindrical, by using Daskalopoulos-Hamilton-\v{S}e\v{s}um's classification of $2$-dimensional closed ancient solutions
\cite{DaskalopoulosHamiltonSesum12}. Combining with Brendle's result,
they concluded that the soliton must be the Bryant soliton. See also \cite{Guo09MA} for some earlier observations in this approach.

When the scalar curvature decays faster than linearly, i.e., $R\in o(r^{-1})$, O. Munteanu,
C.-J. Sung, J. Wang \cite{MunteanuSungWang17} proved that the soliton must be isometric to a quotient of the product manifold of $Cigar$ and the Euclidean space.
Moreover, their result also holds for higher dimensional steady gradient Ricci solitons with nonnegative sectional curvature.
This improves Hamilton's result which says that $R$ cannot decay as fast as $Cr^{-2}$
provided that $Rm\geq 0$ and $Ric>0$ \cite[Theorem 9.44]{ChowLuNi06}.
For more classification results concerning curvature decay, one can consult \cite{Deruelle12,DengZhu15}.

On the other hand, steady gradient Ricci solitons with scalar curvature decaying slowly along some directions,
i.e., $a\in(0,1)$,
are more obscure. Such solitons, if exist, must be $\kappa$-collapsed since H. Guo \cite{Guo09}
showed that $c_1 r^{-1}\leq R\leq c_2 r^{-1}$ is a necessary condition for the noncollapsity.
Moreover, P. Wu \cite{Wu13} showed that, for every steady gradient Ricci soliton, the scalar curvature satisfies $\inf_{B_r}R\leq Cr^{-\frac{1}{2}}$.
Hence our assumption (\ref{ab}) looks reasonable and Main Theorem indeed reveals some facts about the collapsing case.
Hamilton has conjectured the existence of $\kappa$-collapsed steady solitons,
whose scalar curvature might decay slowly in one direction and exponentially in another direction.
For more information about this problem, one can consult a paper of H.-D. Cao and C. He \cite{CaoHe}.
We expect that better understanding of the level set flow could exclude the existence of 3-dimensional $\kappa$-collapsed steady solitons.

We should mention that H. Guo \cite{Guo10} showed that a $3$-dimensional steady
gradient Ricci soliton is rotationally symmetric if the integral of umbilical difference decays very fast, namely,
when $\max_M R$ is normalized to be $1$,
\begin{align*}
\int_{\Sigma(t)}\left(|A|^2-\frac{1}{2}H^2\right)\,\mathrm{d}\sigma\in O(e^{-at})\quad\mbox{for some } a>2,
\end{align*}
where $\Sigma(t):=\{x\in M|f=t\}$. This indeed motivates us to study the decay rate of the
umbilical ratio.

\noindent {\bf{Acknowledgement.}}
The first author would like to thank Prof. Pengfei Guan for helpful discussions during the workshop for Besson's 60th birthday.
The second author is supported by the MOST research grant 103-2115-M-002-013-MY3.

\section{Basic facts and known results}

A triple $(M^n,g,f)$ consisting of a complete smooth manifold $M$ with dimension $n$,
a Riemannian metric $g$ and a smooth potential function $f: M\to \mathbb{R}$ is called a {\it gradient Ricci soliton} if
\begin{align*}
Ric+\mathrm{Hess}(f)=\mu g
\end{align*}
for some $\mu\in\mathbb{R}$.
We say a gradient Ricci soliton is {\it shrinking} if $\mu>0$; {\it steady} if $\mu=0$; and {\it expanding} if $\mu<0$.

It is well-known that closed steady or expanding gradient Ricci solitons are trivial, that is,
their potential functions $f$ must be constants and manifolds are Einstein manifolds. In the $3$-dimensional case,
shrinking gradient Ricci solitons are completely classified by Perelman \cite{Perelman02}, Ni-Wallach
\cite{NiWallach08}, Naber \cite{Naber10}, Chen \cite{Chen09} and Cao-Chen-Zhu \cite{CaoChenZhu08}.
Such a soliton must be $\kappa$-noncollapsed and is isometric to one of $\mathbb{S}^3$,
$\mathbb{R}\times\mathbb{S}^2$, $\mathbb{R}^3$, or their isometric quotients.
The expanding case is more subtle, even the $3$-dimensional $\kappa$-noncollapsed ones with positive curvature are not yet classified.
In this article, we concentrate on steady gradient Ricci solitons.
Remark that our main theorem holds analogously for $3$-dimensional expanding gradient Ricci solitons with positive Ricci curvature,
but we will not go through any details here.

As mentioned in the introduction, Brendle \cite{Brendle13} showed that Bryant soliton is the only $3$-dimensional
$\kappa$-noncollapsed steady gradient Ricci soliton with positive curvature.
In fact, by Hamilton-Ivey-Chen's pinching theorem, a $3$-dimensional ancient solution always has nonnegative curvature.
By the strong maximum principle, the solution either has positive curvature or splits a line.
Combining all these facts, $3$-dimensional $\kappa$-noncollapsed steady gradient Ricci solitons are completely classified.
So the remaining unknown case is the collapsing one.
Although qualitative description of generic ($\kappa$-collapsed) steady gradient Ricci solitons are very few,
there is a significant fact proved by Munteanu and Wang that a steady gradient Ricci soliton is either one-ended or isometric to $\mathbb{R}\times N$,
where $N$ is compact and Ricci flat \cite{MunteanuWang11}.
On the other hand, in dimension $3$, $Cigar \times\mathbb{R}$ (and its isometric quotient) is the only collapsed example known to us.
Combining these results, it seems still hard to place the bet on the existence or nonexistence of $\kappa$-collapsed steady gradient Ricci soliton.

Here we list some basic facts about steady gradient Ricci solitons. These properties will be used in our argument.
\begin{Lma}
Let $(M^n,g,f)$ be a steady gradient Ricci soliton, then we have the following properties.
\begin{itemize}
\item[\rm(a)] {\rm (Trace of the soliton equation)} $R+\Delta f=0$.
\item[\rm(b)] {\rm (Traced second Bianchi identity)}
$\partial_{\alpha}R=2\nabla^{\beta}R_{\alpha\beta}=2R_{\alpha\beta}f^{\beta}\equiv 2Ric(\nabla f, \partial_{\alpha})$.
\item[\rm(c)] {\rm (Bochner formula)} $\Delta R+2|Ric|^2=\langle\nabla R,\nabla f\rangle$.
\item[\rm(d)] {\rm (Hamilton's identity \cite[page 156]{ChowLuNi06})} $R+|\nabla f|^2=C_0$ for some constant $C_0$.
\item[\rm(e)] {\rm (B.-L. Chen \cite{Chen09})} $R\geq 0$. Moreover, when $n=3$, the sectional curvature is positive except for $Cigar\times\mathbb{R}$,
$\mathbb{R}^3$, and their isometric quotients.
\end{itemize}
\end{Lma}

\section{Asymptotically umbilical}
In the following sections, we will concentrate on the $3$-dimensional steady gradient Ricci soliton.

Let $(M^3,g,f)$ be a steady gradient Ricci soliton with positive sectional curvature.
Suppose that the scalar curvature $R$ satisfies the asymptotic behavior $c_1 r^{-b}\leq R\leq c_2 r^{-a}$ for some positive constants $a$ and $b$.
Since $f$ is strictly concave, we know that $f$ achieves its maximum value at some point $O$.
It implies $R(O)=\max_M R=C_0$ and $M$ is diffeomorphic to $\mathbb{R}^3$.
Without loss of generality, we may assume $f(O)=0$.
Moreover, each level set of $f$, denoted by $\Sigma_t:=\{f=t\}$, is diffeomorphic to $\mathbb{S}^2$ by Morse theory.
This structure leads us to consider the level set flow $F:\mathbb{S}^2\times(-\infty,0)\to M$ such that $F(\cdot,t)=\Sigma_t$,
that is,
\begin{align}
\left\{
\begin{array}{l}
\displaystyle\frac{\partial}{\partial t}F(x,t)=\frac{1}{|\nabla f|}\nu=\frac{1}{R-R_{\nu\nu}}H\nu \\[4mm]
F(\cdot,t_0)=\Sigma_{t_0}
\end{array}
\right., \label{LSFeqn}
\end{align}
where $\nu=\frac{\nabla f}{|\nabla f|}$ is the unit inward normal vector, and $R_{\nu\nu}=Ric(\nu,\nu)$.
Let $\lambda=\frac{1}{R-R_{\nu\nu}}$.
By viewing this level set flow as the $\lambda$-weighted mean curvature flow,
we may borrow computations from Huisken's paper \cite{Huisken86} on the mean curvature flow.
Here we list evolution equations of the geometric quantities under the level set flow (\ref{LSFeqn}).
\begin{Prop} \label{evoh}
Choosing local coordinates $\{x^1,x^2\}$ on $\Sigma_t$ such that
$\{t=f, x^1, x^2\}$ forms a local coordinate chart on $M^3$ and denoting $\partial_t=\frac1{|\nabla f|}\nu$,
we have the following equations.
\begin{itemize}
\item[\rm(a)] The evolution equation of the second fundamental form $h_{ij}$ of $\Sigma_t$ in $M^3$ is
\begin{align*}
\frac{\partial}{\partial t} h_{ij}
&=\lambda\left(
\begin{array}{l}
\Delta h_{ij}-2Hh_{il}h^l_{\ j}+ h_{ij}(|A|^2+R_{\nu\nu}) \\[2mm]
-h_{jl}R^{l\ \ \ m}_{\ mi}-h_{il}R^{l\ \ \ m}_{\ mj}+2h_{lm}R^{l\ m}_{\ i\ j}+\nabla_j R_{\nu li}^{\ \ \ l}+\nabla_l R^{\ \ \ l}_{\nu ij}
\end{array}\right) \\
&\quad+H\lambda_{,ij}+H_{,i}\lambda_{,j}+H_{,j}\lambda_{,i}.
\end{align*}
In particular, the mean curvature $H$ of $\Sigma_t$ in $M^3$ satisfies
\begin{align*}
\frac{\partial}{\partial t}H=\lambda\left(\Delta H+H(|A|^2+R_{\nu\nu})\right)+H\Delta\lambda+2\langle\nabla\lambda,\nabla H\rangle.
\end{align*}
\item[\rm(b)] The evolution equation of the norm square of the second fundamental form $|A|^2$ of $\Sigma_t$ in $M^3$ is
\begin{align*}
\hspace*{12mm}\frac{\partial}{\partial t}|A|^2
=\lambda(\Delta |A|^2-2|\nabla A|^2+2|A|^2(|A|^2+R_{\nu\nu})-B)+2h^{ij}\left( H\lambda_{,ij}+2H_{,i}\lambda_{,j}\right),
\end{align*}
where $B=4h^{ij}h_{jl}R_{\ mi\ }^{l\ \ \ m} - 4 h^{ij}h^{lm}R_{iljm}-2h^{ij}(\nabla_j R_{\nu li}^{\ \ \ l}+\nabla_l R_{\nu ij}^{\ \ \ l})$.
\end{itemize}
\end{Prop}
We define the {\it umbilical ratio} $U_\sigma=\frac{2|A|^2-H^2}{H^{2+\sigma}}$ for some constant $\sigma\in\mathbb{R}$
and derive the evolution equation of the umbilical ratio $U_\sigma$ as follows.
\begin{Prop} \label{prop2}
The evolution equation of the umbilical ratio $U_\sigma$ is
\begin{align*}
\frac{\partial}{\partial t} U_\sigma
&=\lambda\left(\Delta U_{\sigma}+\frac{2(1+\sigma)}{H}\langle \nabla H, \nabla U_{\sigma}\rangle
 -\frac{2}{H^{4+\sigma}}|\nabla_iHh_{jk}-H\nabla_ih_{jk}|^2\right)\\
&\quad+\lambda\left(\frac{\sigma(1+\sigma)}{H^2}|\nabla H|^2-\sigma(|A|^2+R_{\nu\nu})\right)U_{\sigma} \\
&\quad-\frac{\lambda}{H^{2+\sigma}}B-(2+\sigma)\left(\Delta\lambda+\frac{2}{H}\langle\nabla H,\nabla\lambda\rangle\right)U_{\sigma}+\frac{2}{H^{2+\sigma}}D,
\end{align*}
where
\begin{align*}
B&=4h^{ij}h_{jl}R_{\ mi\ }^{l\ \ \ m}-4 h^{ij}h^{lm}R_{iljm}-2h^{ij}(\nabla_j R_{\nu li}^{\ \ \ l}+\nabla_l R_{\nu ij}^{\ \ \ l}),\quad\mbox{and} \\
D&=h^{ij}\left(H\lambda_{,ij}+2H_{,i}\lambda_{,j}\right)-\frac{1}{2}H(H\Delta\lambda+2\langle\nabla H,\nabla\lambda\rangle).
\end{align*}
\end{Prop}

Next two Lemmas we will transform the terms $B$ and $D$ into the quantities $\frac{\partial}{\partial t}U_{\sigma},U_{\sigma}$, or $\sqrt{U_{\sigma}}$
so that we can analyze the behavior of the umbilical ratio through the evolution equation.

\begin{Lma} \label{B}
Suppose that $p$ is a non-umbilical point of $\Sigma_t$, then in a neighborhood of $p\in M^3$, we have
\begin{align*}
-\frac{1}{H^{2+\sigma}}B
=|\nabla f|^2\frac{\partial}{\partial t}U_{\sigma}
-\left(2C_0-\frac{2+\sigma}{H}\langle\nabla H,\nabla f\rangle-2H|\nabla f|\right)U_{\sigma}
-\frac{8L_{22}}{H^{\frac{2+\sigma}{2}}|\nabla f|^{3}}\sqrt{U_\sigma},
\end{align*}
where $L(\cdot,\cdot):=2\mathrm{d}R\otimes\mathrm{d}R-|\nabla R|^2g$,
$L_{22}=L(e_2,e_2)$,
$e_2$ is the unit eigenvector of the second fundamental form $h_{ij}$
with respect to the larger eigenvalue,
and $C_0$ is a constant satisfying $R+|\nabla f|^2=C_0$.
\end{Lma}

\begin{proof}
Given $p\in\Sigma_t$,
we can choose local coordinates $\{x^0=f,x^1,x^2\}$ around $p\in M^3$ such that
$\{x^1,x^2\}$ forms normal coordinates around $p\in \Sigma_t$, and $e_1:=\partial_1 ,e_2:=\partial_2$
are eigenvectors of the second fundamental form $h_{ij}$ at $p$ in this local chart (This could be done only when $x^0$ is chosen to be $f$).
This coordinates system is called {\it subnormal} in this paper.
Since $p$ is not an umbilical point, we assume $h_{22}>h_{11}$ in convenience.
In this setting, we have $g^{00}=|\nabla f|^2$, $\nabla f=|\nabla f|^2\frac{\partial}{\partial x^0}$, and $\nu=|\nabla f|\partial_0$.

Let $S^2=2|A|^2-H^2$ be the numerator part of the umbilical ratio $U_\sigma$.
First, we can compute that the first two terms of $B$ will be $4R_{1212}S^2$.
This is because in the subnormal coordinates, we have
\begin{align*}
&\quad\ 4h^{ij}h_{jl}R_{\ mi\ }^{l\ \ \ m}-4 h^{ij}h^{lm}R_{iljm} \\
&=4h_{22}h_{22}R_{2121}+4h_{11}h_{11}R_{1212}
-4h_{11}h_{22}R_{1212}-4h_{22}h_{11}R_{2121} \\
&=4R_{1212}(h_{22}^2+h_{11}^2-2h_{11}h_{22})=4R_{1212}S^2.
\end{align*}
Next, we will compute the rest two terms of $B$ (without minus sign):
\begin{align*}
&\quad\ 2h^{ij}\left(\nabla_j R_{\nu li}^{\ \ \ l}+\nabla_l R_{\nu ij}^{\ \ \ l}\right)\\
&=2\left(h_{11}\nabla_1R_{\nu 212}+h_{22}\nabla_2R_{\nu 121}+h_{11}\nabla_2R_{\nu 112}+h_{22}\nabla_1R_{\nu 221}\right)\\
&=2\left(h_{22}-h_{11}\right)\left(\nabla_2R_{\nu 121}-\nabla_1R_{\nu 212} \right)
=2\left(h_{22}-h_{11}\right)(\nabla_2R_{\nu 2}-\nabla_1R_{\nu 1}).
\end{align*}
Using the Ricci soliton equation, we interchange the orders of derivatives and get
\begin{align*}
&\quad\ \nabla_2R_{\nu 2}-\nabla_1R_{\nu 1} \\
&=-(\nabla_{2}\mathrm{Hess}(f)(\nu,e_2))+(\nabla_{1}\mathrm{Hess}(f)(\nu, e_1)) \\
&=-(\nabla_\nu\mathrm{Hess}(f)(e_2,e_2))-\left\langle R(\nu,e_2)e_2,\nabla f\right\rangle
+(\nabla_\nu\mathrm{Hess}(f)(e_1, e_1))+\left\langle R(\nu,e_1)e_1,\nabla f\right\rangle \\
&=\nabla_\nu R_{22}+R_{2\nu \nu 2}|\nabla f|-\nabla_\nu R_{11}-R_{1\nu \nu 1}|\nabla f| \\
&=\nabla_\nu R_{22}-\nabla_\nu R_{11}+|\nabla f|(R_{2\nu \nu 2}+R_{2112}-R_{2112}-R_{1\nu \nu 1}) \\
&=\nabla_\nu R_{22}-\nabla_\nu R_{11}-|\nabla f|(R_{22}-R_{11}).
\end{align*}
On a neighborhood of a non-umbilical point $p\in M^3$, by the soliton equation, we have $S=h_{22}-h_{11}=\frac{R_{22}-R_{11}}{|\nabla f|}$,
which implies $R_{22}-R_{11}=S|\nabla f|$, and
\begin{align*}
\nu(S|\nabla f|)=\nu(R_{22}-R_{11})
=\nabla_{\nu}R_{22}+2Ric(\nabla_{\nu}e_2,e_2)-\nabla_{\nu}R_{11}-2Ric(\nabla_{\nu}e_1,e_1).
\end{align*}
Up to now, we get the following equality:
\begin{align}
&\quad\ 2h^{ij}\left(\nabla_j R_{\nu li}^{\ \ \ l}+\nabla_l R_{\nu ij}^{\ \ \ l}\right) \notag
=2S(\nabla_2 R_{\nu 2} -\nabla_1R_{\nu 2})\\
&=2S\left(\nu(S|\nabla f|)-2Ric(\nabla_{\nu}e_2,e_2)+2Ric(\nabla_{\nu}e_1,e_1)-S|\nabla f|^2\right). \label{Equation001}
\end{align}

Next, for $i=1,2$, since $\nabla_{\nu}e_i-\nabla_{e_i}\nu=[\nu,e_i]$ and $\nabla_{e_i}\nu=-h_{ii}e_i$, we have at $p$,
\begin{align*}
\nabla_{\nu}e_i
=-h_{ii}e_i+[|\nabla f|\partial_0,e_i]
=-h_{ii}e_i-e_i(|\nabla f|)\partial_0
=-h_{ii}e_i+\frac{e_i(R)}{2|\nabla f|}\partial_0.
\end{align*}
The last equality holds because of $R+|\nabla f|^2=C_0$ (see Lemma 1 (d)), and it implies $e_2(R)=-2|\nabla f|e_2(|\nabla f|)$.
So we have
\begin{align*}
Ric(\nabla_{\nu}e_i,e_i)
&=Ric\left(-h_{ii}e_i+\frac{e_i(R)}{2|\nabla f|}\partial_0,e_i\right)
=-h_{ii}R_{ii}+\frac{e_i(R)}{2|\nabla f|}R_{0i} \\
&=-h_{ii}R_{ii}+\frac{(e_i(R))^2}{4|\nabla f|^3},
\end{align*}
where we use the property $e_i(R)=\partial_iR=2Ric(\nabla f,\partial_i)=2Ric(|\nabla f|^2\partial_0,\partial_i)=2|\nabla f|^2R_{0i}$.
Thus
\begin{align*}
-Ric(\nabla_{\nu}e_2,e_2)+Ric(\nabla_{\nu}e_1,e_1)
&=h_{22}R_{22}-\frac{(e_2(R))^2}{4|\nabla f|^3}
-h_{11}R_{11}+\frac{(e_1(R))^2}{4|\nabla f|^3}.
\end{align*}
We introduce a tensor $L(\cdot,\cdot)=2\mathrm{d}R\otimes\mathrm{d}R-|\nabla R|^2g$
such that $L_{22}=L(e_2,e_2)=(e_2(R))^2-(e_1(R))^2$, and it gives
\begin{align*}
-Ric(\nabla_{\nu}e_2,e_2)+Ric(\nabla_{\nu}e_1,e_1)
=h_{22}R_{22}-h_{11}R_{11}-\frac{4L_{22}}{|\nabla f|^{3}}
=HS|\nabla f|-\frac{4L_{22}}{|\nabla f|^{3}}.
\end{align*}
So equation (\ref{Equation001}) becomes
\begin{align}
&\quad\ 2h^{ij}\left(\nabla_j R_{\nu li}^{\ \ \ l}+\nabla_l R_{\nu ij}^{\ \ \ l}\right) \notag \\
&=2S\left(\nu(S|\nabla f|)+HS|\nabla f|-\frac{4L_{22}}{|\nabla f|^{3}}-S|\nabla f|^2\right) \notag \\
&=|\nabla f|^2(\partial_0S^2)+2S^2\left(|\nabla f|\partial_0(|\nabla f|)+H|\nabla f|-|\nabla f|^2\right)-\frac{8SL_{22}}{|\nabla f|^{3}} \notag \\
&=|\nabla f|^2(\partial_0S^2)+2S^2(-R_{\nu\nu}+H|\nabla f|-|\nabla f|^2)-\frac{8SL_{22}}{|\nabla f|^{3}}. \label{equation002}
\end{align}

On the other hand, since
\begin{align*}
\partial_0U_{\sigma}=\partial_0\left(\frac{S^2}{H^{2+\sigma}}\right)=\frac{\partial_0 S^2}{H^{2+\sigma}}-\frac{2+\sigma}{H}(\partial_0H)U_{\sigma},
\end{align*}
we have
\begin{align}
\frac{|\nabla f|^2}{H^{2+\sigma}}\partial_0 S^2
&=|\nabla f|^2\left(\partial_0 U_{\sigma}+\frac{2+\sigma}{H}(\partial_0H)U_{\sigma}\right) \notag \\
&=|\nabla f|^2(\partial_0U_{\sigma})+\frac{2+\sigma}{H}\langle \nabla H, \nabla f\rangle U_{\sigma}. \label{equation003}
\end{align}
We combine the results (\ref{equation002}) and (\ref{equation003}) to get
\begin{align*}
-\frac{1}{H^{2+\sigma}}B
&=-\frac{4}{H^{2+\sigma}}R_{1212}S^2+\frac{2}{H^{2+\sigma}}h^{ij}\left(\nabla_j R_{\nu li}^{\ \ \ l}+\nabla_l R_{\nu ij}^{\ \ \ l}\right) \\
&=-\frac{4}{H^{2+\sigma}}R_{1212}S^2+|\nabla f|^2(\partial_0U_\sigma)+\frac{2+\sigma}{H}\langle\nabla H,\nabla f\rangle U_{\sigma} \\
&\quad+\frac{2}{H^{2+\sigma}}S^2(-R_{\nu\nu}+H|\nabla f|-|\nabla f|^2)-\frac{8SL_{22}}{H^{2+\sigma}|\nabla f|^{3}} \\
&=|\nabla f|^2(\partial_0U_{\sigma})-\left(4R_{1212}-\frac{2+\sigma}{H}\langle\nabla H,\nabla f\rangle+2R_{\nu\nu}-2H|\nabla f|+2|\nabla f|^2\right)U_{\sigma} \\
&\quad-\frac{8L_{22}}{H^{\frac{2+\sigma}{2}}|\nabla f|^{3}}\sqrt{U_\sigma} \\
&=|\nabla f|^2\frac{\partial}{\partial t}U_{\sigma}
-\left(2C_0-\frac{2+\sigma}{H}\langle\nabla H,\nabla f\rangle-2H|\nabla f|\right)U_{\sigma}
-\frac{8L_{22}}{H^{\frac{2+\sigma}{2}}|\nabla f|^{3}}\sqrt{U_\sigma}.
\end{align*}
\end{proof}

Next Lemma will rewrite the term $D$ in Proposition \ref{prop2} in terms of $\sqrt{U_{\sigma}}$.
\begin{Lma} \label{D}
Suppose that $p$ is a non-umbilical point of $\Sigma_t$, then in a neighborhood of $p\in M^3$, we have
\begin{align*}
D=\frac{H^{\frac{2+\sigma}{2}}}{2}\left(H(\lambda_{,22}-\lambda_{,11})+2(H_{,2}\lambda_{,2}-H_{,1}\lambda_{,1})\right)\sqrt{U_\sigma},
\end{align*}
where $_{,2}$ (resp. $_{,1}$) denotes the covariant differentiation with respect to the unit eigenvector of
$h_{ij}$ which corresponds to the larger (resp. smaller) eigenvalue.
\end{Lma}

\begin{proof}
Using the subnormal coordinates around a given point $p$ as we have done in the proof of Lemma \ref{B}, one can compute
\begin{align*}
D&=h^{ij}(H\lambda_{,ij}+2H_{,j}\lambda_{,i})-\frac{1}{2}(H^2\Delta\lambda+2H\langle\nabla\lambda,\nabla H\rangle)\\
 &=H\left(h_{11}\lambda_{,11}+h_{22}\lambda_{,22}  -\frac{H}{2}( \lambda_{,11} +\lambda_{,22} )\right)\\
 &\quad+2(h_{11}H_{,1}\lambda_{,1}+h_{22}H_{,2}\lambda_{,2})-H(\lambda_{,1}H_{,1}+\lambda_{,2}H_{,2})\\
 &=\frac{H}{2}\left(\lambda_{,11}(h_{11}-h_{22})+\lambda_{,22}(h_{22}-h_{11})\right)+H_{,1}\lambda_{,1}(h_{11}-h_{22})  + H_{,2}\lambda_{,2}(h_{22}-h_{11})\\
 &=\frac{S}{2}\left(H(\lambda_{,22}-\lambda_{,11})+2(H_{,2}\lambda_{,2}-H_{,1}\lambda_{,1})\right).
\end{align*}
Then  the conclusion follows from the relation $U_\sigma=\frac{2|A|^2-H^2}{H^{2+\sigma}}=\frac{S^2}{H^{2+\sigma}}$.
\end{proof}

Combining Proposition~\ref{prop2}, Lemma~\ref{B}, and Lemma~\ref{D}, one obtains
\begin{Prop}\label{prop3}
Under the level set flow (\ref{LSFeqn}),
the umbilical ratio satisfies the following evolution equation:
\begin{align*}
\frac{\partial}{\partial t}U_\sigma
=&\ \lambda\left(\Delta U_{\sigma}+\frac{2(1+\sigma)}{H}\langle\nabla H,\nabla U_{\sigma}\rangle
 -\frac{2}{H^{4+\sigma}}|\nabla_iHh_{jk}-H\nabla_ih_{jk}|^2\right) \\
&+\lambda\left(\frac{\sigma(1+\sigma)}{H^2}|\nabla H|^2-\sigma(|A|^2+R_{\nu\nu})\right)U_\sigma+\lambda|\nabla f|^2\frac{\partial}{\partial t}U_{\sigma}\\
&-\lambda\left(2C_0-\frac{2+\sigma}{H}\langle\nabla H,\nabla f\rangle-2H|\nabla f|\right)U_{\sigma}
-\frac{8\lambda L_{22}}{H^{\frac{2+\sigma}{2}}|\nabla f|^{3}}\sqrt{U_\sigma} \\
&-(2+\sigma)\left(\Delta\lambda+\frac{2}{H}\langle\nabla H,\nabla\lambda\rangle\right)U_{\sigma} \\
&+\left(\frac{1}{H^{\frac{\sigma}{2}}}(\lambda_{,22}-\lambda_{,11})+\frac{2}{H^{\frac{2+\sigma}{2}}}(H_{,2}\lambda_{,2}-H_{,1}\lambda_{,1})\right)\sqrt{U_\sigma},
\end{align*}
where $_{,2}$ (resp. $_{,1}$) denotes the covariant differentiation with respect to the unit eigenvector of $h_{ij}$
which corresponds to the larger (resp. smaller) eigenvalue.
\end{Prop}

Now we are ready to show the following theorem, which leads to the Main Theorem.
A crucial observation is that the leading term of the coefficient of $U_\sigma$ is negative,
so we can apply the maximum principle to the evolution equation of $U_\sigma$ and show that $U_\sigma$ is bounded for some suitably chosen $\sigma$.

\begin{Thm} \label{theorem1}
Let $(M^3, g, f)$ be a nontrivial 3-dimensional steady gradient Ricci soliton.
If the scalar curvature $R$ satisfies
\begin{align*}
c_1 r^{-b}(x)\leq R(x)\leq c_2 r^{-a}(x)
\end{align*}
for some $a\in(0,1], b\geq a$, $c_1,c_2>0$, and $r(x):=dist(O,x)$ for some point $O\in M$,
then the soliton satisfies
\begin{align*}
\frac{2|A|^2-H^2}{H^2}\in O\left(r^{6a-\frac{8a^2}{b}}\right),
\end{align*}
where $|A|^2$ and $H$ are the norm square of the second fundamental form and the mean curvature of
the level sets of $f$, respectively.
\end{Thm}

\begin{proof}[Proof of Theorem~\ref{theorem1}]
As one can see in Proposition \ref{prop3}, there naturally arises a time-derivative of $U_\sigma$,
namely the term $\lambda|\nabla f|^2\frac{\partial}{\partial t}U_\sigma$,
in the right hand side of the evolution equation of $U_\sigma$. Moving it to the left hand side and denoting $\tau=-t$,
and $\Theta=\frac{\lambda}{\lambda|\nabla f|^2-1}$, we get
\begin{align*}
\frac{\partial}{\partial \tau}U_\sigma
=&\ \Theta\left(\Delta U_{\sigma}+\frac{2(1+\sigma)}{H}\langle\nabla H,\nabla U_{\sigma}\rangle
 -\frac{2}{H^{4+\sigma}}|\nabla_iHh_{jk}-H\nabla_ih_{jk}|^2\right) \\
&+\Theta\left(\sigma(1+\sigma)\frac{|\nabla H|^2}{H^2}-\sigma(|A|^2+R_{\nu\nu})\right)U_\sigma\\
&-\Theta\left(2C_0-\frac{2+\sigma}{H}\langle\nabla H,\nabla f\rangle-2H|\nabla f|\right)U_{\sigma}
-\frac{8\Theta L_{22}}{H^{\frac{2+\sigma}{2}}|\nabla f|^{3}}\sqrt{U_\sigma}\\
&- \frac{2+\sigma}{\lambda|\nabla f|^2-1}\left(\Delta\lambda+\frac{2}{H}\langle\nabla H,\nabla\lambda\rangle\right)U_{\sigma} \\
&+ \frac{1}{\lambda|\nabla f|^2-1}\left(\frac{1}{H^{\frac{\sigma}{2}}}(\lambda_{,22}-\lambda_{,11})
+\frac{2}{H^{\frac{2+\sigma}{2}}}(H_{,2}\lambda_{,2}-H_{,1}\lambda_{,1})\right)\sqrt{U_\sigma},
\end{align*}

Recall that the value of $f$ goes to $-\infty$ as $x\to\infty$.
So this equation reveals the asymptotic shape of level sets of $f$ when $\tau$ goes to $\infty$.
Besides, the coefficient of $\Delta U_\sigma$, $\Theta = \frac{1}{|\nabla f|^2-(R-R_{\nu\nu})}$, approaches $\frac{1}{C_0}$ as $\tau\to\infty$,
so the parabolic equation is non-degenerate.
Furthermore, the coefficient of $U_\sigma$ is featured by the term $-2\Theta C_0$ because all other terms fade out as $\tau\to\infty$.
This suggests that $U_\sigma$ is a decreasing quantity and its decay rate is determined by the coefficient of $\sqrt{U_\sigma}$, which consists of
\begin{align*}
{\rm I}:=-\frac{8\Theta L_{22}}{ H^{\frac{2+\sigma}{2}}|\nabla f|^{3}},
\hspace{3mm}{\rm I\!I}:=\frac{\lambda_{,22}-\lambda_{,11}}{(\lambda|\nabla f|^2-1)H^{\frac{\sigma}{2}}},
\ \mbox{ and }\ {\rm I\!I\!I}:=\frac{2(H_{,2}\lambda_{,2}-H_{,1}\lambda_{,1})}{(\lambda|\nabla f|^2-1)H^{\frac{2+\sigma}{2}}}.
\end{align*}
Since $c_1 r^{-b}(x)\leq R(x)\leq c_2 r^{-a}(x)$,
by Shi's estimate, we know that $L_{22}\in O(r^{-3a})$,
$|\nabla \lambda| \in O(r^{2b-\frac{3}{2}a})$, and $\left|\nabla^2\lambda\right| \in O(r^{3b-3a})$.
Thus,
\begin{align*}
{\rm I} \in O(r^{\frac{2+\sigma}{2}b -3a}), \ \ {\rm I\!I} \in O(r^{\frac{6+\sigma}{2}b-4a}) \
\mbox{ and }\ {\rm I\!I\!I} \in O(r^{\frac{6+\sigma}{2}b -4a}).
\end{align*}
Since $b\geq a$, it is easy to see that ${\rm I\!I}$ and ${\rm I\!I\!I}$ are the highest order terms.
By choosing $\sigma=\frac{8a}{b}-6$, we have ${\rm I\!I\!I}\in O(1)$.
Therefore,
\begin{align*}
\frac{\partial}{\partial \tau} U_\sigma
\leq\frac{1}{C_0}\left(\Delta U_{\sigma}+\frac{2(1+\sigma)}{H}\langle\nabla H,\nabla U_{\sigma}\rangle\right)-U_{\sigma}+C\sqrt{U_\sigma}
\end{align*}
By the maximum principle, $U_\sigma$ is bounded when $\tau\to\infty$ and thus
\begin{align*}
\frac{2|A|^2-H^2}{H^2}\in O(r^{-a\sigma})=O\left(r^{6a-\frac{8a^2}{b}}\right).
\end{align*}
\end{proof}

\begin{proof}[Proof of Main Theorem]
Let $a\in(0,1]$. Main Theorem is just a combination of Theorem~\ref{theorem1} and the
following fact derived from Shi's estimate:
\begin{align*}
U_0=\frac{(R_{22}-R_{11})^2}{4|\nabla f|^2H^2}=\frac{(R_{2\nu 2\nu}-R_{1\nu 1\nu})^2}{4|\nabla f|^2H^2}\in O(r^{2b-4a}),
\end{align*}
since $|R_{j\nu j\nu} |\leq |R_{\nu\nu}|\in O(r^{-2a})$.
\end{proof}

\bibliography{refer_article}
\bibliographystyle{alpha}
\end{document}